\documentclass[11pt,a4paper]{amsart}
\usepackage{amsfonts,amsmath,amssymb,amsthm}
\usepackage{times}
\usepackage{graphics,graphicx}
\usepackage{color}
\usepackage{enumerate}

\usepackage[utf8]{inputenc}
\usepackage{bbm, dsfont}
\usepackage{mathtools}
\mathtoolsset{showonlyrefs=true}
\usepackage{subfig}

\renewcommand{\epsilon}{\varepsilon}

\newcommand{\pP}{\mathbb{P}}
\newcommand{\cH}{\mathcal{H}}
\newcommand{\cC}{\mathcal{C}}
\newcommand{\cM}{\mathcal{M}}
\newcommand{\E}{\mathbb{E}}
\newcommand{\N}{\mathbb{N}}
\newcommand{\Z}{\mathbb{Z}}
\newcommand{\cE}{{\mathcal E}}

\newtheorem{theorem}{Theorem}[section]{\bf}{\it}
{\bf}{\it}
\newtheorem{proposition}[theorem]{Proposition}{\bf}{\it}
\newtheorem{corollary}[theorem]{Corollary}{\bf}{\it}
\newtheorem{lemma}[theorem]{Lemma}{\bf}{\it}

\title[The dominating colour of an infinite Pólya urn model]{The dominating colour of an infinite Pólya urn model}
\author[Erik ~Th\"ornblad]{Erik Th\"ornblad}
 \address{Department of Mathematics, Uppsala University, Box 480, S-75106 Uppsala, Sweden.}
 \email{erik.thornblad@math.uu.se}
 \date{\today}
\begin{document}
% \begin{abstract}
%   We study the growth rate of a clique in a random graph model defined as follows. At time $0$, start with an isolated vertex. Then, choose a vertex uniformly at random. With probability $p>1/2$, add a new vertex and connect it to the chosen vertex and all its neighbour. With probability $1-p$, remove all the edges of the chosen vertex. This graph consists of disjoint cliques. Using a new two--colour P\'{o}lya--type urn model with random replacement matrix, we show that, almost surely, there is a single clique that is eventually the largest. Using martingale methods we then show that a fixed clique grows like $N_m^{1/\beta}$, where $\beta=\frac{p}{2p-1}$ and $N_m$ is the number of vertices at time $m$. Since there is a single clique that is eventually the largest, this implies that the maximum degree of the graph grows like $N_m^{1/\beta}$.
% \end{abstract}

\begin{abstract}
 We study a P\'{o}lya--type urn model defined as follows. Start at time $0$ with a single ball of some colour. Then, at each time $n\geq 1$, choose a ball from the urn uniformly at random. With probability $1/2<p<1$, return the ball to the urn along with another ball of the same colour. With probability $1-p$, recolour the ball to a new colour and then return it to the urn. This is equivalent to the supercritical case of a random graph model studied by Backhausz and M\'{o}ri \cite{BackhauszMori, BackhauszMori2015} and Thörnblad \cite{Thornblad2015}. We prove that, with probability $1$, there is a dominating colour, in the sense that, after some random but finite time, there is a colour that always has the most number of balls. A crucial part of the proof is the analysis of an urn model with two colours, in which the observed ball is returned to the urn along with another ball of the same colour with probability $p$, and removed with probability $1-p$. Our results here generalise a classical result about the P\'{o}lya urn model (which corresponds to $p=1$).

\noindent \textbf{Keywords:} urn model, largest colour, random graphs, persistent hub.

\noindent \textbf{AMS subject classification:} 60G50, 60J80.
%\keywords{Urn Model, Dominating Colour, Random graphs, Persistent hub}
%\subjclass[2010]{60G50, 60J80}
\end{abstract}
\maketitle

%%%%%%%%%%%%%%%%%%%%%%%%%%%%%%%%%%%%%%%%%%%%%%%%%%%%%%%%%%%%%%%%%%%%%%%%
\section{Introduction}
%%%%%%%%%%%%%%%%%%%%%%%%%%%%%%%%%%%%%%%%%%%%%%%%%%%%%%%%%%%%%%%%%%%%%%%%
We study an urn model described as follows. At time $0$, start with a single ball of some colour. At each time step $n\geq 1$, choose a ball uniformly at random. 

\begin{enumerate}
 \item With probability $p$, return the ball to the urn along with another ball of the same colour.
  \item With probability $1-p$, recolour the ball with a new colour and then return it to the urn.
\end{enumerate} 
In this paper we will typically consider the case $p>1/2$, although the definition or the urn model makes sense for any $0\leq p \leq 1$. 
This urn model has a (countably) infinite number of colours. It also allows for the extinction of colours. If the last ball of a certain colour is recoloured, then this colour will never appear again in the urn. It is equivalent to the following random graph model studied in \cite{BackhauszMori, BackhauszMori2015, Thornblad2015}. Let $G_0$ be the graph with a single isolated vertex. Create $G_n$ from $G_{n-1}$ by doing one of the following steps.
\begin{enumerate}
 \item With probability $p$, do a \emph{duplication step}. Select a clique in $G_{n-1}$ with probability proportional to size, and introduce a new vertex to that clique.
  \item With probability $1-p$, do a \emph{deletion step}. Select a clique in $G_{n-1}$ with probability proportional to size and delete a vertex from the chosen clique. Then introduce a new clique with a single vertex.
\end{enumerate}
The equivalence of these models is clear once we identify each clique with a colour in the urn. 

Let us mention a few known results about the urn model, coming from \cite{BackhauszMori, BackhauszMori2015, Thornblad2015}. These results were originally proved in the random graph version, but transfer immediately to the urn version. The degree distribution is known to exhibit a phase transition from exponential decay to power law in the three regimes $0<p<1/2, p=1/2$ and $1/2<p<1$, referred to as the subcritical, critical and supercritical case, respectively. Knowledge of the degree distribution of the random graph model translates to knowing the almost sure limits of the quantities $\frac{U_{j,n}}{N_n}$, where $U_{j,n}$ is the number of colours at time $n$ with $j$ balls, and $N_n$ is the number of balls at time $n$. This was done for $p=1/2$ in \cite{BackhauszMori} and for the remaining cases in \cite{Thornblad2015}. Both exact and asymptotic results were found. Later Backhausz and M\'{o}ri \cite{BackhauszMori2015} revisited the model and determined bounds on the logarithmic growth rate of the maximal clique size of the graph, i.e. the number of balls of the leading colour. In particular, in the supercritical case $p>1/2$ they found that
\begin{align}
 \frac{p}{\beta}\leq \liminf_{n\to \infty}\frac{\log M_n}{\log N_n}\leq \liminf_{n\to \infty}\frac{\log M_n}{\log N_n}\leq  \frac{1}{\beta}, \label{eq:BM}
\end{align}
where $\beta=\frac{p}{2p-1}$, $M_n$ is the size of the leading colour at time $n$, and $N_n$ is the number of balls in the urn at time $n$. As we shall see later, this result can be strengthened to show that $M_n\sim \mu N_n^{1/\beta}$ for some random variable $\mu>0$, implying that the upper bound in \eqref{eq:BM} is the correct one. This was indeed the correct growth rate conjectured in \cite{BackhauszMori2015}. We remark that this result was originally put in terms of the maximal degree, which is equal to one less than the maximal clique size, which by the identification is equal to the number of balls of the leading colour.

Similar studies have been done for population models. Champagnat and Lambert \cite{ChampagnatLambert2012b} studied a population model in which individuals were given i.i.d. lifetime distributions and give birth at constant rate. Furthermore, individuals then mutate at constant rate and change allelic type. If the lifetime distribution has a unit point mass as $\infty$, births are at rate $p$ and mutations at rate $1-p$, then this corresponds to a continuous--time version of our model, where colours correspond to the alleles. Champagnat and Lambert achieved a number of convergence results, in all three regimes, about the oldest and most abundant families, mainly in expectation and distribution. By constrast, we shall derive almost sure results, but only in the supercritical regime. 

One of our main results is that, provided $p>1/2$, then, with probability $1$, there is some colour that after some random but finite time becomes \emph{dominant}, i.e. remains the colour with the most number of balls forever. A similar problem was studied by Khanin and Khanin \cite{KhaninKhanin2001}. They consider an urn model with $k$ colours, with a parameter $r>0$. Balls are added sequentially, and the probability of adding a ball of colour $1$ is proportional to $x^r$, where $x$ is the number of balls of colour $1$, etc. They show that for $r>1/2$ one of the colours will be eventually dominant almost surely, but for $r\leq 1/2$ the colours change leadership infinitely many times. Indeed, for $r>1$, all but finitely many of the new balls are of the same colour, creating a \emph{monopoly} of one of the colours. Similar results were achieved by Chung, Handjani and Jungreis \cite{ChungHandjaniJungreis2003} for an infinite urn model. This model works as follows (slightly rephrased to allow for more direct comparison to our model). With probability $p$, add a ball, the colour of which is chosen like in the Khanin/Khanin--model. With probability $1-p$, add a ball of a new colour. The difference to our urn model is that colours can never lose balls in the Chung/Handjani/Jungreis--model. Also, the number of balls in the Chung/Handjani/Jungreis--model grows deterministically, which makes analysis slightly easier. Although these models appear similar, qualititively they behave rather differently. For $r=1$ it is found that the number of colours of size $k$ in the Chung/Handjani/Jungreis--model decays like a power--law for all $0<p<1$. This should be contrasted with our model, where there is a phase transition from exponential decay to power law at $p=1/2$.

In the random graph interpretation of our model, the notion of a dominant colour should be seen as a variation of the notion of a \emph{persistent hub}, a concept considered by Dereich and Mörters \cite{DereichMorters2009} and by Galashin \cite{Galashin2014}. The latter considered the classical preferential attachment model and a variation thereof and showed that, almost surely, there is a vertex that after some random but finite time (and always thereafter) is the vertex of maximal degree in the graph. This vertex is called the persistent hub. As remarked in \cite{BackhauszMori2015}, our random graph model does not have a persistent hub on the level of vertices, since all vertices will be selected for deletion infinitely often, thus pushing their degree down to $0$. However, by using the correspondence between vertices in a clique and balls of a certain colour, we will see that there is a clique that almost surely is the largest one, i.e. a persistent clique--hub.

%We mention also Oliviera and Spencer  \cite{OlivieraSpencer2005}, who considered a super--linear preferential attachment tree and showed that almost surely there is a single vertex that eventually recieves all new vertices. This of course is a very strong form of the persistent hub, and corresponds to the notion of monopoly mentioned above.

The rest of this paper is outlined as follows. We will typically embed the discrete urn model in a corresponding continuous--time model. First we use the contraction method to determine the growth rate of a fixed colour. Then, to show that there is an eventually dominating colour, we follow the approach taken by Galashin \cite{Galashin2014}. His methods extrapolate to our setting without any significant changes. We start by observing the joint behaviour of two fixed colours and determine exactly the asymptotic distribution of the quantity $\frac{B_n}{W_n+B_n}$, where $B_n$ and $W_n$ are the number of balls of the two colours at time $n$, respectively. From this distribution we deduce that two colours can change relative leadership at most finitely many times. Furthermore, we can determine exactly the probability that one of the colours ever overtakes the other, which allows us to bound the probability that a colour with only one ball (a new colour) will overtake a colour with many balls (the leading colour). This probability will turn out to be sufficiently small, in the sense that we can apply the Borel--Cantelli lemma to show that colours with few balls overtake the currently leading colour only a finite number of times, with probability $1$. This, along with the fact that two fixed colours overtake each other at most a finite number of times, implies the existence of a dominating colour. This dominating colour must grow like some fixed colour, so we are able to determine the asymptotic growth rate of the largest colour of the urn. 

In line with \cite{BackhauszMori2015, Thornblad2015}, let us introduce the notation $\beta=\frac{p}{2p-1}$ and $\gamma=\frac{1-p}{p}$. We use the notation $a_n\stackrel{a.s.}{\sim} b_n$ for (possibly random) sequences $(a_n)_{n=1}^{\infty}$ and $(b_n)_{n=1}^{\infty}$ to mean $\lim_{i\to \infty}\frac{a_n}{b_n}\stackrel{a.s.}{=}1$ and $X\sim F$ to denote that the random variable $X$ has distribution $F$.

%%%%%%%%%%%%%%%%%%%%%%%%%%%%%%%%%%%%%%%%%%%%%%%%%%%%%%%%%%%%%%%%%%%
\section{Results}
%%%%%%%%%%%%%%%%%%%%%%%%%%%%%%%%%%%%%%%%%%%%%%%%%%%%%%%%%%%%%%%%%%%%
Inspired by Athreya's embedding scheme, see \cite[V.9]{AthreyaNeyBook}, we shall embed the discrete urn scheme in a continuous--time urn scheme. The continuous urn scheme is defined as follow. Each ball in the urn has two exponential clocks, one ringing at rate $1/2<p<1$ and the other at rate $1-p$. If the first clock rings, add to the urn another ball of the same colour. If the second clock rings, remove the ball from the urn and add a ball of a new colour. It is easy to see that the discrete process has the same transition probabilities as the continuous process. Moreover, whenever a new colour is created, it behaves like a birth--death process with birth rate $p$ and death rate $1-p$. We characterise the growth rate of such a birth--death process in the next lemma.

\begin{lemma}\label{lem:onecolour}
Let $(X(t))_{t\geq 0}$ be a continuous--time birth--death process with birth rates $p$ and death rates $1-p$. Then
\begin{align}
\frac{X(t)}{e^{(2p-1)(t)}}&\xrightarrow{a.s.} U,
\end{align}
where $U$ is a random variable with distribution $(1-\gamma)\Gamma\left(1,\frac{1}{\beta}\right)+\gamma \delta_0$.
\end{lemma}
Note that $\delta_0$ denotes the distribution that places unit mass at $0$. The quantity $\gamma$ is the extinction probability (which can be found in many different ways).  After some work, one can show that this follows from the results in \cite[III.5]{AthreyaNeyBook}. We instead use the contraction method, which exploits the recursive structure of the process. We sketch the argument here, and refer the reader to \cite{RoslerRuschendorf2001, Ruschendorf2006} for further details and references.

\begin{proof}
 Let $\tau$ be the first event time in the process $X(t)$. With probability $p$, this event is a birth, and with probability $1-p$, it is a death (which then forces $X(t)=0$ for all $t> \tau$). This leads to the distributional equality
 \begin{align}\label{eq:dist}
  X(t) \stackrel{d}{=} \mathbbm{1}_{\{\tau \leq t\}}Y\left(X'(t-\tau)+X''(t-\tau) \right) + \mathbbm{1}_{\{\tau > t\}},
 \end{align}
where $X(t), X'(t)$ and $X''(t)$ are independent and identically distributed, $Y\sim \text{Ber}(p)$ and $\tau\sim \text{Exp}(1)$. Normalising we obtain
 \begin{align}
   \frac{X(t)}{e^{(2p-1)t}} \stackrel{d}{=} e^{-(2p-1)\tau}\mathbbm{1}_{\{\tau \leq t\}}Y \left(\frac{X'(t-\tau)}{e^{(2p-1)(t-\tau)}}+\frac{X''(t-\tau)}{e^{(2p-1)(t-\tau)}} \right) + \frac{1}{e^{(2p-1)t}}\mathbbm{1}_{\{\tau > t\}}.
  \end{align}
 Note that $\mathbbm{1}_{\{\tau > t\}} \stackrel{a.s.}{\to} 0 $ as $t\to \infty$; similarly $\mathbbm{1}_{\{\tau \leq t\}}\stackrel{a.s.}{\to}1 $ as $t\to \infty$. It can be shown that $X(t)/e^{(2p-1)t}$ is a non--negative martingale, so by the martingale convergence theorem it converges almost surely to some random variable $U$, which then must satisfy the distributional equality
 \begin{align}
  U \stackrel{d}{=} e^{-(2p-1)\tau}Y\left(U'+U'' \right), \label{eq:Udist}
 \end{align}
 where $U,U',U''$ are independent and identically distributed, $Y\sim \text{Ber}(p)$ and $\tau \sim \text{Exp}(1)$. 

Consider the space $\cM$ of distributions with finite second moment and first moment equal to $1$, equipped with the Wasserstein metric
\begin{align}
 d(\lambda_1,\lambda_2)=\inf \|Z_1-Z_2 \|_2 
\end{align}
where the infimum is over all random variables $Z_1$ and $Z_2$ with $Z_1\sim \lambda_1$ and $Z_2\sim \lambda_2$ and $\| \cdot \|_2$ denotes the $L_2$--norm. Let $T:\cM\to \cM$ be the distributional operator $TZ \stackrel{d}{=}e^{-(2p-1)\tau}Y\left(Z'+Z'' \right)$, with $\tau, Y$ independent and like above, and $Z',Z''$ independent and distributed like $Z$. Note in particular that $TZ\in \cM$ for any $Z\in \cM$, so this operator is well--defined. It can be shown that $\E[2e^{-2(2p-1)\tau}Y^2]=\sqrt{\frac{2p}{4p-1}}<1$ is a contracting factor in the Wasserstein metric for the operator $T$. By the Banach fixed point theorem, it follows that the \eqref{eq:Udist} has a unique solution in $\cM$, see e.g. \cite[Theorem 2.2]{Ruschendorf2006}. Moreover, it can be verified that the distribution $(1-\gamma)\Gamma(1,1/\beta)+\gamma\delta_0$ lies in $\cM$ and satisfies $\eqref{eq:Udist}$, so this must be the unique solution. It suffices now to show that the distribution of $U$ lies in $\cM$, i.e. that $\E[U]=1$ and $\E[U^2]<\infty$. It follows from \cite[III.4--5]{AthreyaNeyBook} that
\begin{align}
 \text{Var}(X(t))=\frac{e^{2(2p-1)t}(1-e^{-(2p-1)t})}{2p-1},
\end{align}
which implies that $\text{Var}(X(t)/e^{(2p-1)t})\leq \frac{1}{2p-1}$ for all $t\geq 0$ and in particular that the second moment of the limiting random variable $U$ is finite. Therefore the first moment of $U$ is also finite, and it follows that  $\E[U]=1$ since $U$ is the limit of a martingale sequence with expectation one. Therefore the distribution of the limiting random variable must be in $\cM$, so the contraction method shows that $U\sim (1-\gamma)\Gamma(1,1/\beta)+\gamma\delta_0$.

\end{proof}

Using this we can relate the growth rate of a fixed colour to the growth rate of the entire urn.

\begin{theorem}\label{thm:asymp}
 Let $X_n$ be the size of a fixed colour at time $n$. Conditional on survival of this colour, there exists a positive random variable $\nu$ such that
\begin{align}
 X_n \stackrel{a.s.}{\sim} \nu N_n^{1/\beta}.
\end{align}
\end{theorem}

\begin{proof}
We make use of the continuous--time embedding described earlier. For this purpose, let $X(t)$ be the number of balls of a fixed colour at time $t$ in the corresponding continuous model. Suppose this colour was born at time $t_0$. By Lemma \ref{lem:onecolour}, conditional on survival, we have $X(t)\stackrel{a.s.}{\sim}  e^{(2p-1)(t-t_0)}U$, where $U\sim \Gamma(1,1/\beta)$. Denoting by $N(t)$ the total number of balls by time $t$, similarly one can show that  $N(t) \stackrel{a.s.}{\sim} e^{pt} V$ where $V\sim \Gamma(1,1)$. This is done in \cite{BackhauszMori2015}. Therefore, conditional on survival of the fixed colour, we have $X(t)\sim \nu N(t)^{\frac{2p-1}{p}}$ almost surely, where $\nu$ is a positive random variable (that depends on $t_0, U$ and $V$). 

Now, let $T_n$ be the $n$:th birth or death of the process $(X(t))_{t\geq 0}$. Observing the values $(X(T_n))_{n=1}^{\infty}$ gives us the discrete urn process, so it suffices to show that $T_n\xrightarrow{a.s} \infty$ as $n\to \infty$.  This is well--known if $p=1$, see e.g. \cite[III]{AthreyaNeyBook}, and is equivalent to showing that the process does not explode in finite time. But we can couple our birth--death process $X(t)$ for $1/2<p<1$ to a pure birth process with $p=1$, in such a way that it grows slower. Hence it also cannot explode in finite time, so $T_n\xrightarrow{a.s} \infty$ on the event of non--extinction. Recall now that $\frac{2p-1}{p}=\frac{1}{\beta}$. Then $X_n\stackrel{a.s.}{\sim} \nu N_n^{1/\beta}$.
\end{proof}

As mentioned, a colour survives with probability $1-\gamma$ and goes extinct with probability $\gamma$, so this gives us the following corollary.

\begin{corollary}
 Let $X_n$ be the size of a fixed colour at time $n$. Then there exists a positive random variable $\nu>0$ such that
 \begin{align}
 X_n \stackrel{a.s.}{\sim} \begin{cases} 0, & \text{ with probability } \gamma ,\\
           \nu N_n^{1/\beta}, & \text{ with probability } 1-\gamma.
          \end{cases}
\end{align}
\end{corollary}

Let us now turn to the joint behaviour of two fixed colours. To do this, we consider the projection of the entire urn onto two colours. That is, we study an urn model with balls of two colours, black and white say. Initially there are $B_0=b$ black balls and $W_0=w$ white balls. Sequentially sample balls from the urn, uniformly at random. With probability $1/2<p<1$, return the ball to the urn with a ball of the same colour, and with probability $1-p$ remove the urn from the urn with negative probability. To avoid degeneracies we stop the evolution of the urn if one of the colours disappear from the urn. Conditional on $B_n\neq 0, W_n\neq 0$, the transition probabilities are given by
\begin{align}
 (B_{n+1},W_{n+1})=
\begin{cases}
 (B_n+1,W_n) & \text{ with probability } p\frac{B_n}{B_n+W_n} \\
 (B_n,W_n+1)& \text{ with probability } p\frac{W_n}{B_n+W_n} \\
 (B_n-1,W_n) & \text{ with probability } (1-p)\frac{B_n}{B_n+W_n}\\
 (B_n,W_n-1) & \text{ with probability } (1-p)\frac{W_n}{B_n+W_n} \\
\end{cases}
\end{align}
A similar urn model was studied in \cite{Isanovaetal}, where the urn came with an additional immigration procedure to ensure that both colours survive. We mention the paper \cite{Janson2004}, that to a large extent solved the case of so--called \emph{tenable} urn models. Our urn model does not fall into this category (in particular the condition of irreducibility is not satisfied, i.e. that any configuration of the urn can be reached from any starting configuration). However, the process  $(B_n,W_n)_{n=0}^{\infty}$ can be seen as a triangular urn scheme with random replacement matrix, allowing for non--negative entries. Triangular urn schemes with deterministic replacement matrix were considered by Janson \cite{Janson2005}, the results of which were extended by Aguech \cite{Aguech2009} to random triangular replacement matrices. The latter however imposed that the random variables be non--negative, to ensure the survival of the urn. Our urn model does not have non--negative entries; however, the condition $p>1/2$ implies that the replacement distribution has positive expectation, so the colours (and the urn) survive with positive probability.

We are interested in the joint behaviour of two colours, or more precisely the proportion
\begin{align}
f_{b,w}(n)= \frac{B_n}{W_n+B_n}.
\end{align}
As mentioned, we stop whenever $f_{b,w}(n)=0$ or $f_{b,w}(n)=1$, so these are absorbing states of the process $(f_{b,w}(n))_{n=0}^{\infty}$. Alternatively, instead of stopping the process here, we could condition on being on the event of non--extinction, i.e. that not all balls disappear from the urn. The probability of this event, which is positive, will appear implicitly later on.

We shall again use the continuous time embedding to evaluate the evolution of the urn. That is, we consider indepndent black and white birth--death process $(B_i(t))_{i=1}^{b}$ and  $(W_i(t))_{i=1}^{w}$  started at $B_i(0)=W_i(0)=1$, with birth rates $p$ and death rates $1-p$. It is again easy to show that the discrete process $(B_n,W_n)_{n=0}^{\infty}$has the same transition probabilities as the continuous process $\left(\sum_{i=1}^{b}B_i(t),\sum_{i=1}^{w}W_i(t) \right) _{t\geq 0}$, so we will study the quantity
\begin{align}
 g_{b,w}(t)=\frac{\sum_{i=1}^{b}B_i(t)}{\sum_{i=1}^{b}B_i(t)+\sum_{i=1}^{w}W_i(t)}
\end{align}
instead of its discrete analogue $f_{b,w}(n)$. Again we stop the process if we ever reach $g_{b,w}(t)=0$ or $g_{b,w}(t)=1$. Since the processes are equivalent, we will be able to show that almost sure convergence of $g_{b,w}(t)$ implies almost sure convergence of $f_{b,w}(n)$ to the same limit.

% Namely, suppose we observe the urn at time $n_0$, and that the urn contains $B_{n_0}=b$ black balls, $W_{n_0}=w$ white balls and $R_{n_0}=r$ red balls. We use the continuous time embedding described earlier to consider evaluate the future evolution of the urn. That is, we consider independent black and white birth--death processes processes $(B_i(t))_{i=1}^{b}$ and  $(W_i(t))_{i=1}^{w}$  started at $B_i(0)=W_i(0)=1$, with birth rates $p$ and death rates $1-p$, and additionally a red process $(R_i(t))_{i=1}^{r}$, independent from the other, which is a birth process with rate $p$, with immigration of a new particle whenever a black or white process experiences a death. It is easy to show that the discrete process $(B_n,W_n,R_n)_{n=0}^{\infty}$ has the same transition probabilities as the continuous process $\left(\sum_{i=1}^{b}B_i(t),\sum_{i=1}^{w}W_i(t),\sum_{i=1}^{r}R_i(t) \right) _{t\geq 0}$, so we will study the quantity
% \begin{align}
%  g_{b,w}(t)=\frac{\sum_{i=1}^{b}B_i(t)}{\sum_{i=1}^{b}B_i(t)+\sum_{i=1}^{w}W_i(t)}
% \end{align}
% instead of its discrete analogue $f_{b,w}(n)$. Again we stop the process if we ever reach $g_{b,w}(t)=0$ or $g_{b,w}(t)=1$. Since the processes are equivalent, we will be able to show that almost sure convergence of $g_{b,w}(t)$ implies almost sure convergence of $f_{b,w}(n)$ to the same limit.

\begin{proposition}\label{prop:limdist}
The limit $\lim_{t\to \infty}g_{b,w}(t)$ exists almost surely, and its distribution is given by the mixture 
\begin{align}
 r_{b,w}\delta_0 + r^{\ast}_{b,w}\text{Beta}(G_b,H_w) + r_{w,b}\delta_1,
\end{align}
where $r_{b,w}=\gamma^b\left(1-\frac{b}{b+w}\gamma^{w}\right)$, $r_{w,b}=\gamma^w\left(1-\frac{w}{b+w}\gamma^{b}\right)$, $r^{\ast}_{b,w}=(1-\gamma^b)(1-\gamma^w)$, and $G_b, H_w$  are independent discrete random variables with probability mass functions
\begin{align}
 \pP[G_b=k]&=\frac{1}{1-\gamma^{b}}\binom{b}{k}(1-\gamma)^k \gamma^{b-k},  &k=1,\dots, b, \\
 \pP[H_w=k]&=\frac{1}{1-\gamma^{w}}\binom{w}{k}(1-\gamma)^k \gamma^{w-k},  &k=1,\dots, w,
\end{align}
i.e. binomial random variables conditioned on not being zero.
\end{proposition}

\begin{proof}
Let $\tau_b=\inf\{t \ : \ \sum_{i=1}^b B_i(t)=0 \}$ and $\sigma_w = \inf\{t \ : \ \sum_{i=1}^{w} W_i(t)=0 \}$ be the extinction times of the two processes. The event that the white process dies out is $\{\tau_w<\infty\}$, and similar for the black process. Note in particular that these events are independent.

On the event $\{\tau_b<\sigma_w<\infty\}\cup \{\tau_b<\sigma_w=\infty\}$ we have $g_{b,w}(\tau_b)=0$. The first event occurs with probability $\frac{w}{w+b}\gamma^{b+w}$, since the probability that one of the $w$ white processes is the last one to die is $\frac{w}{w+b}$ (conditional on all $b+w$ processes dying), by symmetry, and both processes die with probability $\gamma^{b+w}$, by independence. The second event occurs with probability $\gamma^b(1-\gamma^w)$, and the sum of these probabilities is $r_{b,w}$. By symmetry, the probability of the event $\{\sigma_w<\tau_b<\infty\}\cup \{\tau_b<\sigma_w=\infty\}$ is $r_{w,b}$.

It is clear that $g_{b,w}(\tau_b)=0$ on $\{\tau_b<\sigma_w<\infty\}\cup \{\tau_b<\sigma_w=\infty\}$. Since $0$ is an absorbing state, this gives a point mass at $0$ with relative mass $r_{b,w}$. Similarly $g_{b,w}(\sigma_w)=1$ on $\{\sigma_w<\tau_b<\infty\}\cup \{\tau_b<\sigma_w=\infty\}$, giving a point mass at $1$ with relative mass $r_{w,b}$. 

The event that neither process dies out (so at least one white and at least one black process survives) is the event $\{\tau_b=\infty, \sigma_w=\infty \}$. By independence this has probability
\begin{align}
 \pP[\tau_b=\infty, \sigma_w=\infty]=(1-\gamma^b)(1-\gamma^w)=:r^{\ast}_{b,w}.
\end{align}

Let $G_b,H_w$ be as in the statement. On the event $\{\tau_b=\infty, \sigma_w=\infty \}$ we have at least one surviving process of each kind. Conditional on this, since each black process survives with probability $1-\gamma$ independently of the other, the number of surviving black processes is distributed like $G_b$. Similarly the number of surviving white processes on this event is distributed like $H_w$. Recall the scaling limit in Lemma \ref{lem:onecolour}, which holds for each surviving process separately. Therefore, on this event,
\begin{align}
 g_{b,w}(t)=\frac{e^{-(2p-1)t}\sum_{k=1}^{G_b}B_{k_i}(t)}{e^{-(2p-1)t}\sum_{k=1}^{G_b}B_{k_i}(t)+e^{-(2p-1)t}\sum_{k=1}^{H_w}W_{k_i}(t)} \xrightarrow{a.s.} \frac{U}{U+V},
\end{align}
as $t\to \infty$, where $U\sim \Gamma(G_b,1/\beta), V\sim \Gamma(H_w,1/\beta)$. But then $\frac{U}{U+V}\sim \text{Beta}(G_b,H_w)$.
\end{proof}

\begin{figure}
\centering
\subfloat{\includegraphics[width=0.45\textwidth]{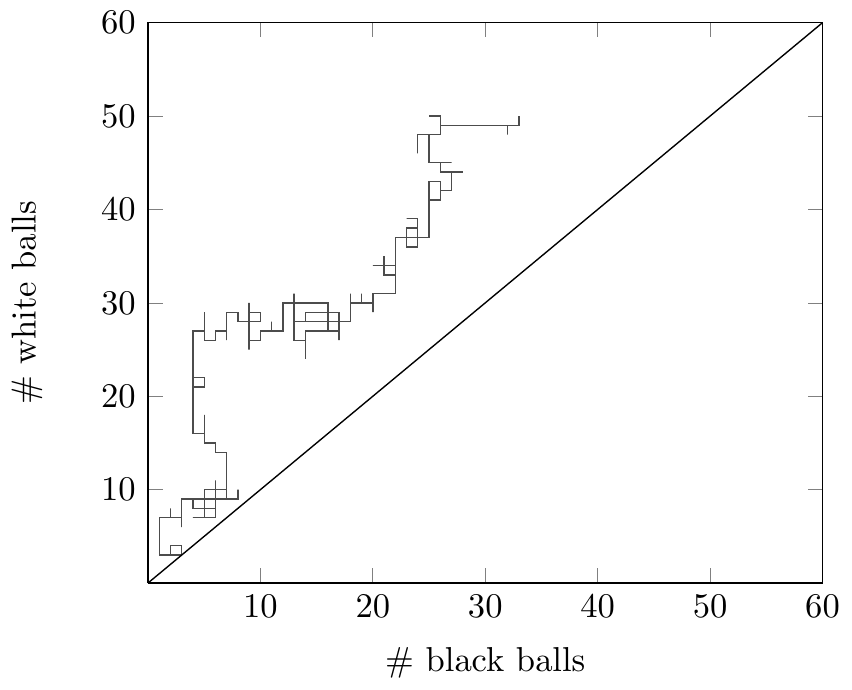}} 
\qquad 
\subfloat{\includegraphics[width=0.45\textwidth]{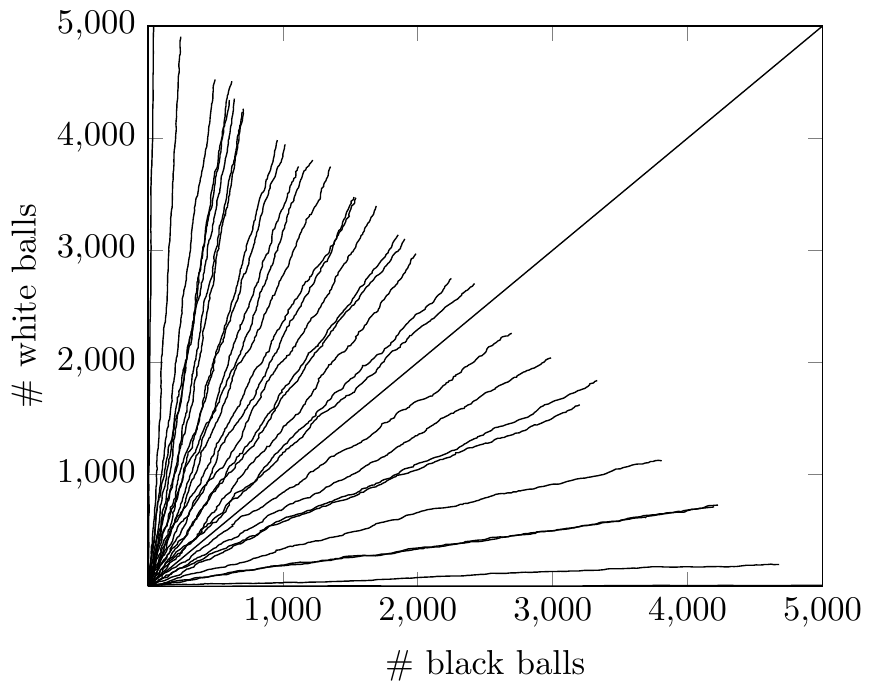}}
\caption{The left--most figure is a simulation of a single instance of the urn model with $(b,w)=(2,3)$ and $p=3/5$, with $300$ drawings. The second one is a simulation of 30 trials with 20 000 drawings each, and parameters $(b,w)=(2,3)$ and $p=3/4$. In particular, note that diagonal--crossing and absorption tends to happen early, if at all. The simulations were done in MatLab.}
\label{fig1}
\end{figure}

Proposition \ref{prop:limdist} immediately transfers to the discrete--time urn model.

\begin{theorem}\label{thm:disclimdist}
 With the same notation as in Proposition \ref{prop:limdist}, $\lim_{n\to \infty}f_{b,w}(n)$ exists almost surely, with distribution given by the mixture
\begin{align}
 r_{b,w} \delta_0 + r^{\ast}_{b,w}\text{Beta}(G_b,H_w) + r_{w,b}\delta_1.
\end{align}
\end{theorem}

\begin{proof}
Let $T_n$ be the time of the $n$:th birth or death in the process $\left(\sum_{i=1}^{b}B_i(t),\sum_{i=1}^{w}W_i(t) \right)$, conditional on survival. Since $f_{b,w}(n)=g_{n,w}(T_n)$, it suffices to show that $T_n\xrightarrow{a.s.} \infty$ as $n\to \infty$. The proof of this is similar to the argument in the proof of Lemma \ref{lem:onecolour}. 
\end{proof}

It is natural to extend the above results to $k$ colours rather than $2$. This corresponds to projecting our infinite urn model onto $k$ fixed colours. Namely, consider an urn with $k$ colours, and let $X_i(n)$ denote the number of balls of colour $i$ at time $n$, with initial condition $X_i(0)>0$. At each time step, draw a ball uniformly at random. With probability $p$, replace it to the urn along with a ball of the same colour. With probability $1-p$, remove it from the urn. Let $S(n)=X_1(n)+\cdots X_k(n)$. With the same method as above, it is not difficult to show, on the event of non--extinction of the urn, that
\begin{align}
 (X_1(n)/S(n),\dots, X_k(n)/S(n)) \stackrel{d}{\to} (Y_1,\dots, Y_k)
\end{align}
where each $Y_i$ is distributed as a mixture of Dirichlet distributions (possibly degenerate to subspaces), the parameters of which will be binomially distributed conditioned on not all being zero. We leave the details for the reader.

We remark that $p=1$ gives us the original P\'{o}lya urn model. Indeed, for $p=1$ we have $\gamma=0$, and it is readily checked that the limit in Theorem \ref{thm:disclimdist} reduces to $\text{Beta}(b,w)$, which is a classical result in the literature. We also point out that all birth--death processes go extinct almost surely in the case $p\leq 1/2$, so we would have convergence of $g_{b,w}(t)$ and $f_{b,w}(n)$ to a convex combination of two point masses at $0$ and $1$. However, even such urns have been studied in the literature, see e.g. \cite{KubaPanholzer2012} which studies a class of \emph{diminishing} urn processes. For instance, our urn model with $p=0$ may be seen as sampling without replacement.

Let us return to the supercritical case $p>1/2$ and concentrate on the event that there is an equal number of black and white balls in the urn. It is easy to show that this occurs at most finitely many times.

\begin{corollary}\label{cor:finchanges}
With probability $1$, the event $B_n=W_n$ occurs for at most finitely many $n$.
\end{corollary}

\begin{proof}
For all possible realisations of $G_b,H_w$, the distribution $\text{Beta}(G_b,H_w)$ is absolutely continuous on $(0,1)$, so the mixture in Theorem \ref{thm:disclimdist} is also absolutely continuous on $(0,1)$. Therefore $\frac{B_n}{W_n+B_n}$ converges almost surely to some constant $c\neq 1/2$. But this means that the fraction equals $1/2$ at most a finite number of times.
\end{proof}

Knowing that $B_n=W_n$ for at most finitely many $n$, it is natural to ask what the probability is that this event occurs at all. Galashin \cite{Galashin2014} and Antal, Ben--Naim and Krapivsky \cite{AntalBenNaimKrapivsky2010} considered this problem for $p=1$, i.e. the classical P\'{o}lya urn model. They view the process $(B_n,W_n)$ as a random walk on $\Z^2$. For $p=1$ this process can only go right or up, so it is possible to count the number of lattice paths from a given starting point to a fixed point on the diagonal. Using the fact that these paths are exchangeable and summing over all points on the diagonal, it is possible to bound the probability that the process ever hits the diagonal. For $p<1$ there are infinitely many paths to any point on the diagonal, so this approach is not possible here, but an argument given by Wallstrom \cite{Wallstrom2012} for the case $p=1$ can easily be adapted to our setting. 

\begin{proposition}\label{prop:eqprob}
Suppose $b>w$. Using the same notation as in Proposition \ref{prop:limdist}, let $F_{b,w}$ be the distribution function of the mixture $r_{b,w} \delta_0 + r^{\ast}_{b,w}\text{Beta}(G_b,H_w) + r_{w,b}\delta_1$. Let $P(b,w)$ denote the probability that $\frac{B_n}{B_n+W_n}=\frac{1}{2}$ for some $n$. Then
\begin{align}
 P(b,w)=2F_{b,w}(1/2).
\end{align}
\end{proposition}

\begin{proof}
Let $\varphi$ be the random limit of $f_{b,w}(n)=\frac{B_n}{B_n+W_n}$, and let $\cE=\inf \{n\geq 0 \ : \ f_{b,w}(n)=\frac{1}{2}\}$. Since $\pP[\varphi=1/2]=0$, we have $\{\cE<\infty\}=\{\cE<\infty, \varphi>1/2\} \cup \{\cE<\infty, \varphi<1/2\}$. The process $(B_n,W_n)_{n=\cE+1}^{\infty}$ is Markovian, so it depends only on $(B_{\cE},W_{\cE})$. Thus, on the event $\cE<\infty$, the limit $\varphi$ is chosen according to $F_{B_{\cE},B_{\cE}}$, which is symmetrical around $1/2$. Therefore $\pP[\cE<\infty, \varphi<1/2]=\pP[\cE<\infty, \varphi>1/2]$. But since $b>w$ we have that $\{\varphi<1/2\}\subseteq \{\cE<\infty\}$, so $\pP[\cE<\infty,\varphi<1/2]=\pP[\varphi<1/2]$. Thus
\begin{align}
 P(b,w)=\pP[\cE<\infty] = 2\pP[\varphi<1/2]=2F_{b,w}(1/2).
\end{align}
\end{proof}

Later on it will be useful to have a bound on the probability $P(b,1)$ for large $b$. This will gives us the probability that a new colour (which starts with a single ball) in the infinite--colour urn model ever catches up with an old colour. In the next lemma we show that this probability decreases at least exponentially in $b$.
\begin{lemma}\label{lem:ineq}
The equalisation probability $P(b,1)$ satisfies the inequality
\begin{align}
 P(b,1)\leq 2\left(\frac{1}{2p}\right)^b
\end{align}
\end{lemma}

\begin{proof}
By Proposition \ref{prop:eqprob} we have
\begin{align}
 P(b,1)&=2F_{b,1}(1/2)\\
&=2\gamma^{b}\left(1-\frac{b}{b+1} \gamma \right) \\
& \qquad  +2(1-\gamma^{b})(1-\gamma)\sum_{k=1}^{b} k\int_0^{1/2}x^{k-1}(1-x)^{1-1}dx \binom{b}{k}\frac{(1-\gamma)^k\gamma^{b-k}}{1-\gamma^{b}} \\
&\leq 2\gamma^{b}+ 2\sum_{k=1}^{b}\frac{1}{2^k}\binom{b}{k}(1-\gamma)^k\gamma^{b-k} \\
&= 2\sum_{k=0}^{b}\binom{b}{k}\left( \frac{1-\gamma}{2} \right)^k\gamma^{b-k} \\
&=2\left(\frac{1-\gamma}{2}+\gamma\right)^b \\
&=2 \left(\frac{1}{2p}\right)^b.
\end{align}
\end{proof}

We now finally return to the urn model with an infinite number of colours. The bound on $P(b,1)$ gives us good enough control over the probability that a small colour ever overtakes a large colour. Using this bound and the Borel--Cantelli lemma, we show now that there can have been at most a finite number of colours that were ever the leader. The proof relies on Lemma \ref{lem:ineq} and the growth rate in Lemma \ref{lem:onecolour}. 

\begin{proposition}\label{prop:finhubs}
 Almost surely there can have been at most finitely many leading colours.
\end{proposition}

\begin{proof}
Let $(T_n)_{n=1}^{\infty}$ be the birth times of new colours. Let $\cH_n$ be the event that the colour born at time $T_n$ ever becomes as large as the leading colour at time $T_n$. The joint behaviour of  the sizes of a new colour and the currently leading colour is described by the $2$--dimensional urn model above started from $(M_n,1)$. Now, for any $r\in\N$, we let $\cC_r$ be the event that $M_n \geq n^{1/2\beta}$ for all $n>r$. Note that this implies that $M_n\geq r^{-1/2\beta}n^{1/2\beta}$ for all $n\geq 1$. For each fixed $r$, we have 
\begin{align}
 \pP[\cH_n \cap \cC_r]
\leq \sup_{A\geq r^{-1/2\beta}n^{1/(2\beta)}}P(A,1) \leq  2\left(\frac{1}{2p} \right)^{r^{-1/2\beta}n^{1/2\beta}}.
\end{align}
 
 \begin{align}
  \sum_{n=1}^{\infty} \pP[\cH_n\cap \cC_r]\leq 2\sum_{n=1}^{\infty} \left(\frac{1}{2p}\right)^{r^{-1/2\beta}n^{1/2\beta}}<\infty,
 \end{align}
where the sum converges by e.g. integral comparison. The Borel--Cantelli lemma implies that $\cH_n\cap \cC_r$ occurs for infinitely many $n$ with probability $0$. Recall Theorem \ref{thm:asymp}, i.e. that the number of balls of any fixed colour, conditional on survival, grows like $\nu N_n^{1/\beta}$, where $\nu$ is some positive random variable. Therefore $\pP[\cC_r]\to 1$ as $r\to \infty$, which implies that $\cH_n$ occurs for infinitely many $n$ with probability $0$. Therefore, with probability $1$, only finitely many colours can have been leaders.
\end{proof}

\begin{theorem} \label{thm:persistenthub}
 Almost surely there is a dominating colour.
\end{theorem}

\begin{proof}
 By Proposition \ref{prop:finhubs} there can have at most a finite number of colours of maximal size, and by Corollary \ref{cor:finchanges} these can have changed leader at most finitely many times. Therefore, with probability $1$, there is a colour that is always the largest after some random but finite time.
\end{proof}

In the corresponding random graph model, this says that almost surely there is a persistent clique--hub. Recall that $M_n$ denotes the size of the leading colour at time $n$ and $N_n$ the total number of balls in the urn at time $n$. Since there almost surely is a dominating colour and we know the growth rate of any fixed colour by Theorem \ref{thm:asymp} (in particular that of the dominating colour), we obtain the following theorem.
\begin{theorem}
 There is a random variable $\mu>0$ such that
\begin{align}
 M_n \stackrel{a.s.}{\sim} \mu N_n^{1/\beta}.
\end{align}
\end{theorem}
By taking logarithms, this also implies the strengthening of the result of Backhausz and M\'{o}ri mentioned in the introduction.

\section{Acknowledgements}
The author thanks Svante Janson for suggesting numerous improvements to the present manuscript, and Katja Gabrysch and C\'{e}cile Mailler for helpful discussions.

\bibliographystyle{abbrv}

\bibliography{urnmodelfinal}

\begin{thebibliography}{10}

\bibitem{Aguech2009}
R.~Aguech.
\newblock Limit theorems for random triangular urn schemes.
\newblock {\em Journal of Applied Probability}, 46(3):827--843, 2009.

\bibitem{AntalBenNaimKrapivsky2010}
T.~Antal, E.~Ben-Naim, and P.~Krapivsky.
\newblock First-passage properties of the {P}{\'o}lya urn process.
\newblock {\em J. Stat. Mech. Theory Exp.}, 7, 2010.

\bibitem{AthreyaNeyBook}
K.~Athreya and P.~E. Ney.
\newblock {\em Branching {P}rocesses}, volume 196 of {\em Die {G}rundlehren der
  mathematischen {W}issenschaften in {E}inzeldarstellungen}.
\newblock Springer, 1972.

\bibitem{BackhauszMori}
{\'A}.~Backhausz and T.~M{\'o}ri.
\newblock Asymptotic properties of a random graph with duplications.
\newblock {\tt arXiv:1308.1506v2}, 2013.
\newblock To appear in Journal of Applied Probability.

\bibitem{BackhauszMori2015}
{\'A}.~Backhausz and T.~M{\'o}ri.
\newblock Further properties of a random graph with duplications and deletions.
\newblock {\tt arXiv:1409.5279v1}, 2015.
\newblock Preprint.

\bibitem{ChampagnatLambert2012b}
N.~Champagnat, A.~Lambert, and M.~Richard.
\newblock Birth and death processes with neutral mutations.
\newblock {\em International Journal of Stochastic Analysis}, 2012, 2012.

\bibitem{ChungHandjaniJungreis2003}
F.~Chung, S.~Handjani, and D.~Jungreis.
\newblock Generalizations of {P}{\'o}lya's urn problem.
\newblock {\em Annals of Combinatorics}, 7:141--153, 2003.

\bibitem{Galashin2014}
P.~Galashin.
\newblock Existence of a persistent hub in the convex preferential attachment
  model.
\newblock arxiv:1310.7513v3, 2014.

\bibitem{Isanovaetal}
A.~Ivanova, S.~D. Durham, W.~F. Rosenberger, and N.~Flournoy.
\newblock A birth and death urn for randomized clinical trials: asymptotic
  methods.
\newblock {\em Sankhy\={a}: The Indian Journal of Statistics, Series B},
  62(1):104--118, 2000.

\bibitem{Janson2004}
S.~Janson.
\newblock Functional limit theorems for multitype branching processes and
  generalized {P}{\'o}lya urns.
\newblock {\em Stochastic Processes and their Applications}, 110(2):177--245,
  2004.

\bibitem{Janson2005}
S.~Janson.
\newblock Limit theorems for triangular urn schemes.
\newblock {\em Probability {T}heory and {R}elated {F}ields}, 134:417--452,
  2005.

\bibitem{KhaninKhanin2001}
K.~Khanin and R.~Khanin.
\newblock A probabilistic model for the establishment of neuron polarity.
\newblock {\em J. Math. Biol.}, 42(1):26--40, 2001.

\bibitem{KubaPanholzer2012}
M.~Kuba and A.~Panholzer.
\newblock Limiting distributions for a class of diminishing urn models.
\newblock {\em Advances in Applied Probability}, 44(1):87--116, 2012.

\bibitem{DereichMorters2009}
P.~M{\"o}rters and S.~Dereich.
\newblock Random networks with sublinear preferential attachment: {D}egree
  evolutions.
\newblock {\em Electronic Journal of Probability}, 14:1222--1267, 2009.

\bibitem{RoslerRuschendorf2001}
U.~R{\"o}sler and L.~R{\"u}schendorf.
\newblock The contraction method for recursive algorithms.
\newblock {\em Algorithmica}, 29(1-2):3--33, 2001.

\bibitem{Ruschendorf2006}
L.~R{\"u}schendorf.
\newblock On stochastic recursive equations of sum- and max-type.
\newblock {\em Journal of Applied Probability}, 43(3):687--703, 2006.

\bibitem{Thornblad2015}
E.~Th{\"o}rnblad.
\newblock Asymptotic degree distribution of a duplication-deletion random graph
  model.
\newblock {\em Internet Math.}, 11(3):289--305, 2015.

\bibitem{Wallstrom2012}
T.~C. Wallstrom.
\newblock The equalization probability of the {P}{\'o}lya urn.
\newblock {\em The American Mathematical Monthly}, 119(6):516--518, 2012.

\end{thebibliography}
\end{document}